\documentclass[12pt, a4paper, twoside]{amsart}

\usepackage[T1]{fontenc}
\usepackage[applemac]{inputenc}
\usepackage{amsmath}
\usepackage{amsthm}
\usepackage{amsfonts}
\usepackage{bbm}
\usepackage[UKenglish]{babel}
\usepackage{enumerate}
\usepackage{bm}
\usepackage{ae}
\usepackage{amssymb}

\theoremstyle{definition}

\theoremstyle{remark}

\theoremstyle{plain}
\newtheorem{thm}{Theorem}[section]
\newtheorem{lem}[thm]{Lemma}

\makeatletter
\newcommand*{\house}[1]{%
  \mathord{%
    \mathpalette\@house{#1}%
  }%
}
\newcommand*{\@house}[2]{%
  \dimen@=\fontdimen8 %
      \ifx#1\scriptscriptstyle\scriptscriptfont
      \else\ifx#1\scriptstyle\scriptfont
      \else\textfont\fi\fi
      3 %
  \sbox0{%
    $#1%
      \vrule width\dimen@\relax
      \overline{%
        \kern2\dimen@
        \begingroup 
          #2%
        \endgroup
        \kern2\dimen@
      }%
      \vrule width\dimen@\relax
      \mathsurround=1.5\dimen@ 
    $%
  }%
  \ht0=\dimexpr\ht0-\dimen@\relax
  \dp0=\dimexpr\dp0+2\dimen@\relax
  \vbox{%
    \kern\dimen@ 
    \copy0 %
  }%
}
\makeatother

\begin{document}

\title[Series of reciprocals]{Arithmetic properties of series of reciprocals of algebraic integers}

\author{S. B. ANDERSEN}

\address{S. B. Andersen, Department of Mathematics, Aarhus University, Ny Munkegade 118,
  DK-8000 Aarhus C, Denmark}

\email{simon\_andersen94@hotmail.com}

\author{S. KRISTENSEN}

\address{S. Kristensen, Department of Mathematics, Aarhus University, Ny Munkegade 118,
  DK-8000 Aarhus C, Denmark}

\email{sik@math.au.dk}

\begin{abstract}
We obtain results bounding the degree of the series $\sum_{n=1}^{\infty} 1/\alpha_n$, where $\{\alpha_n\}$ is a sequence of  algebraic integers satisfying certain algebraic conditions and growth conditions. Our results extend results of Erd\H{o}s, Han\v{c}l and Nair.
\end{abstract}

\maketitle

\section{Introduction}

Questions on the irrationality and transcendence of infinite series have a long history. Often, the series have rational terms, and many famous open problems arise in this way. For instance, the Riemann $\zeta$-function is known to take transcendental values at even, positive integers and to be irrational at $3$. On the other hand, not much is known about the values at any other specified odd integer greater than $1$.

The $\zeta$-values are instances of the family of series
\begin{equation}
\label{eq:series}
\sum_{n=1}^{\infty} \frac{1}{\alpha_n},
\end{equation}
where $\alpha_n$ is an algebraic integer. Concretely, to obtain $\zeta(m)$, we let $\alpha_n = n^m$ in \eqref{eq:series}.

This makes it pertinent to study conditions on the sequence $\{\alpha_n\}$ which ensure irrationality or even transcendence of the series in \eqref{eq:series}. It was shown by Erd\H{o}s \cite{MR539489} that if the $\alpha_n$ are rational integers, $\alpha_n = a_n$, then the series of \eqref{eq:series} is irrational provided $\lim_{n \rightarrow \infty} a_n^{1/2^n} =\infty$. This condition of course falls short of saying anything at all about the $\zeta$-values. Nevertheless, it is essentially best possible, see \cite{MR3705763} for an example. The condition was weakened by Han\v{c}l \cite{MR2061123}, who proved that irrationality of the series \eqref{eq:series} is ensured by the condition 
$$
1 < \liminf_{n \rightarrow \infty}a_n^{1/2^n} <  \limsup_{n \rightarrow \infty}a_n^{1/2^n},
$$
where the $a_n$ are still rational integers.

It appears that not much is known when the $\alpha_n$ are no longer assumed to be rational integers. One result which does exist is due to Han\v{c}l and Nair \cite{MR3705763}. They showed that for $\alpha_n = \sqrt{a_n}$ with $a_n \in \mathbb{N}$, irrationality of the series \eqref{eq:series} is ensured by the condition
$$
\lim_{n \rightarrow \infty} a_n^{2^{-n^2/2}} = \infty.
$$

In the present paper, we develop the methods of Han\v{c}l and Nair to deal with the more general case where the $\alpha_n$ are algebraic integers such that $\vert \alpha_n \vert= \house{\alpha_n}$. Here $\house{\alpha}$ denotes the house of $\alpha$, i.e. the maximum modulus of the conjugates of $\alpha$. Examples for which this holds include Salem and Pisot numbers, but also $d$'th roots of rational integers, $\alpha_n = \sqrt[d]{a_n}$. As we will see, these are in a sense extremal cases for our results, and the former gives better bounds than the latter.

In addition to the irrationality results, the approach taken allows us to give lower bounds on the degree of the series \eqref{eq:series}, where we set the degree of a transcendental number equal to $\infty$.

\begin{thm}
\label{thm:main}
Let $d, D \in \mathbb{N}$, $\epsilon > 0$ and let $\{\alpha_n\}$ be a sequence of algebraic integers with $\max \deg \alpha_n = d$, such that $\vert \alpha_n \vert = \house{\alpha_n}$, such that $\vert \alpha_n\vert$ increases, such that $\vert \alpha_n \vert \ge n^{1+\epsilon}$ for $n$ sufficiently large, such that $\Re (\alpha_n) > 0$ or $\Im(\alpha_n) > 0$ for all $m \in \mathbb{N}$, and such that
$$
\limsup_{n \rightarrow \infty} \house{\alpha_n}^{\frac{1}{D^n \prod_{i=1}^{n-1}(d^i + d)}} = \infty.
$$
Then,
$$
\deg\left(\sum_{n=1}^\infty \frac{1}{\alpha_n}\right) > D.
$$
\end{thm}

Note that if we let $D=d=1$, the conditions of the theorem imply that $\alpha_n$ is a sequence of integers, such that 
$$
\limsup_{n \rightarrow \infty} \vert \alpha_n\vert^{\frac{1}{2^{n-1}}} = \infty.
$$
If the $\alpha_n= a_n$ are assumed to be natural numbers in this case, we recover the result of Erd\H{o}s on noting that 
$$
\limsup_{n \rightarrow \infty} a_n^{\frac{1}{2^{n-1}}} = \sqrt{\limsup_{n \rightarrow \infty} a_n^{\frac{1}{2^{n}}}}.
$$
Similarly, we may recover the result of Han\v{c}l and Nair by letting $D = 1$ and $\alpha_n = \sqrt{a_n}$ with their growth condition on the $a_n$.

Many of the assumptions made in Theorem \ref{thm:main} can be weakened at the cost of making the growth condition on the houses of the $\alpha_n$ more complicated. In the final section of the paper, we will comment on the various conditions and the extent to which they can be weakened.

\section{Auxiliary results}

We will make heavy use of Weil heights and Mahler measures of algebraic numbers. We recall the definitions. 

Let $\alpha$ be an algebraic number, let $K$ be a number field containing $\alpha$ and let $M_K$ denote the set of places of $K$. Then, the (Weil) height of $\alpha$ is defined as
$$
H(\alpha) = \prod_{\nu \in M_K} \max\{1,\vert \alpha \vert_\nu\}^{d_\nu/d},
$$
where $d = [K : \mathbb{Q}]$ and $d_\nu = [K_\nu : \mathbb{Q}_\nu]$, and where $K_\nu$ and $\mathbb{Q}_\nu$ denote the completions of the fields at the place $\nu$. With the normalisation in the exponent, the height becomes independent of the field $K$.

We will also need the Mahler measure of $\alpha$. For this purpose, suppose that $\alpha$ is algebraic of degree $d$ and let $\alpha_1 = \alpha, \alpha_2, \dots, \alpha_d$ denote the conjugates of $\alpha$. Finally, let $a_d$ denote the leading coefficient of the minimal polynomial of $\alpha$ defined over $\mathbb{Z}$. The Mahler measure of $\alpha$ is defined as
$$
M(\alpha) = \vert a_d \vert \prod_{i=1}^d \max\{1, \vert \alpha_i \vert\}.
$$
Here, the only place playing a role is the usual Archimedean one, i.e. the modulus in the complex plane.

The following wonderful result is classical, see e.g. \cite{MR1756786}.

\begin{thm}
\label{thm:height_measure}
For an algebraic number $\alpha$ of degree $d$,
$$
H(\alpha) = M(\alpha)^{1/d}.
$$
\end{thm}

We would further like to relate the house of an algebraic integer $\alpha$ to the height of $\alpha$. The following lemma accomplishes this.

\begin{lem}
\label{lem:height_house_measure}
Let $\alpha$ be an algebraic integer of degree $d$. Then,
$$
H(\alpha) = M(\alpha)^{1/d} \le \house{\alpha} \le M(\alpha) = H(\alpha)^d.
$$
The inequalities are best possible.
\end{lem}

\begin{proof}
The outermost equalities follow immediately from Theorem \ref{thm:height_measure}. For the inequalities, note first that since $\alpha$ is assumed to be an algebraic integer, $a_d = 1$, so that 
$$
M(\alpha) = \prod_{i=1}^d \max\{1, \vert \alpha \vert\}.
$$
Note also that $\house{\alpha} \ge 1$, as $\alpha$ is an algebraic integer. The first inequality now just states that the geometric mean of a set of positive reals is upper bounded by the maximum value, while the second one is even more trivial.

To see that the inequalities are best possible, consider first $\alpha = \sqrt[d]{a}$ with $a$ a positive integer. The minimal polynomial of $\alpha$ is (a factor of) $P(X) = X^d - a$, so that all conjugates of $\alpha$ are of the form $\zeta_d \sqrt[d]{a}$, where $\zeta_d$ is a $d$'th root of unity. Hence, $M(\alpha) = (\sqrt[d]{a})^d = a$. On the other hand, evidently $\house{\alpha} = \sqrt[d]{a}$, so that the first inequality of the lemma is an equality.

Conversely, suppose that $\alpha$ is a Salem or Pisot number of degree $d$. Then, $M(\alpha) = \house{\alpha} = \alpha$, which proves optimality of the second inequality.
\end{proof}

The optimality of the inequalities above justify the remarks in the introduction about the extremal cases of our results. If we are able to get a better-than-expected bound on the Mahler measure of certain numbers in our construction, the resulting growth conditions can be significantly weakened. Our main theorem uses only the lower bound, which is valid in general. However, it should be clear from the remainder of the paper how the arguments may be modified if one has a series of reciprocals of Pisot and Salem numbers.

We will need to know that the height remains unchanged on taking the reciprocal. This is also classical, see \cite{MR1756786}.

\begin{lem}
\label{lem:reciprocal}
Let $\alpha$ be a non-zero algebraic number. Then, $H(\alpha) = H(1/\alpha)$.
\end{lem}

We will  be using the following bounds on heights and degrees of sums of algebraic numbers.

\begin{lem}
\label{lem:elementary}
Let $n \in \mathbb{N}$, and let $\beta_1, \dots, \beta_n$ be algebraic numbers. Then,
$$
H\left(\sum_{i=1}^n \beta_i\right) \le 2^n \prod_{i=1}^n H(\beta_i),
$$
and
$$
\deg\left(\sum_{i=1}^n \beta_i\right) \le \prod_{i=1}^n \deg(\beta_i).
$$
\end{lem}
For a proof of the first inequality, see \cite{MR1756786}. The second inequality can be found in  \cite{MR0258803}.

A final ingredient from the theory of algebraic numbers needed, is the Liouville--Mignotte \cite{liouville, MR554376} bounds on the distance between algebraic numbers. For a unified proof of the inequalities, see appendix A of \cite{MR2136100}

\begin{lem}
\label{lem:liouville}
Let $\alpha$ and $\beta$ be non-conjugate algebraic numbers. Then,
$$
\vert \alpha - \beta \vert \ge \frac{1}{2^{\deg(\alpha)\deg(\beta)}M(\alpha)^{\deg(\beta)}M(\beta)^{\deg(\alpha)}}.
$$
\end{lem}

The final auxiliary results needed are essentially exercises in calculus. The first is the following, essentially proved in \cite{MR539489}.

\begin{lem}
\label{lem:calculus1}
Let $\{a_n\}$ be an increasing sequence of real numbers satisfying for some $\epsilon>0$, that $a_n > n^{1+\epsilon}$ for all $n \in \mathbb{N}$. Then,
$$
\sum_{n=k}^\infty \frac{1}{a_n} < \frac{2+1/\epsilon}{a_k^{\epsilon/(1+\epsilon)}}.
$$
\end{lem}

Note that in Erd\H{o}s' proof, the numerator on the right hand side is left as an unspecified constant depending on $\epsilon$. We give a short proof that this constant may be replaced by $2+1/\epsilon$.

\begin{proof}
By assumption, $a_k^{1/(1+\epsilon)} > k$. Hence,
$$
\sum_{n=k}^\infty \frac{1}{a_n} = \sum_{n=k}^{[a_k^{1/(1+\epsilon)}]} \frac{1}{a_n} + \sum_{n=[a_k^{1/(1+\epsilon)} + 1]}^\infty \frac{1}{a_n} \le \frac{1}{a_k} \sum_{n=k}^{[a_k^{1/(1+\epsilon)}]} 1 + \sum_{n=[a_k^{1/(1+\epsilon)} + 1]}^\infty \frac{1}{n^{1+\epsilon}}. 
$$
For the first sum,
$$
\frac{1}{a_k} \sum_{n=k}^{[a_k^{1/(1+\epsilon)}]} 1 < \frac{a_k^{1/(1+\epsilon)}}{a_k} = \frac{1}{a_k^{\epsilon/(1+\epsilon)}}.
$$
The second sum is estimated by the first summand plus an integral. First,
$$
\sum_{n=[a_k^{1/(1+\epsilon)} + 1]}^\infty \frac{1}{n^{1+\epsilon}} = \frac{1}{[a_k^{1/(1+\epsilon)} + 1]^{1+\epsilon}} + \sum_{n=[a_k^{1/(1+\epsilon)} + 1] +1}^\infty \frac{1}{n^{1+\epsilon}}.
$$
Since $\epsilon/(1+\epsilon) < 1$, 
$$
a_k^{\epsilon/(1+ \epsilon)} < a_k \le [a_k^{1/(1+\epsilon)} + 1]^{1+\epsilon},
$$
so that
$$
\frac{1}{[a_k^{1/(1+\epsilon)} + 1]^{1+\epsilon}} < \frac{1}{a_k^{\epsilon/(1+\epsilon)}}.
$$
For the final sum,
$$
\sum_{n=[a_k^{1/(1+\epsilon)} + 1]+1}^\infty \frac{1}{n^{1+\epsilon}} < \int_{[a_k^{1/(1+\epsilon)} + 1]}^\infty \frac{\textnormal{d}x}{x^{1+\epsilon}} = \frac{1}{\epsilon ([a_k^{1/(1+\epsilon)} + 1])^\epsilon} \le \frac{1/\epsilon}{a_k^{\epsilon/(1+\epsilon)}}.
$$
The result follows on combining the inequalities.
\end{proof}

The second result is extracted from the proof of the main theorem in \cite{MR3705763}, where special cases are used a number of times. We deduce a general form here.

\begin{lem}
\label{lem:calculus2}
Let $\{a_n\}_{n=1}^\infty$ be a sequence of real numbers such that 
$$
\limsup_{n \rightarrow \infty} a_n = \infty.
$$
Then for infinitely many $k \in \mathbb{N}$,
$$
a_{k+1} > \left(1+\frac{1}{k^2}\right) \max_{1 \le n \le k} a_n.
$$
\end{lem}

\begin{proof}
Suppose to the contrary that an $N_0 \in \mathbb{N}$ exists, such that for $k \ge N_0$, 
$$
a_{k+1} \le \left(1+\frac{1}{k^2}\right) \max_{1 \le n \le k} a_n.
$$
Then, for any $k \ge N_0$, 
\begin{alignat*}{2}
a_{k+1} &\le  \left(1+\frac{1}{k^2}\right) \max_{1 \le n \le k} a_n \\
&\le \left(1+\frac{1}{k^2}\right)\left(1+\frac{1}{(k-1)^2}\right) \max_{1 \le n \le k-1} a_n \\
&\le \cdots \\
&\le \prod_{n=N_0+1}^k \left(1+\frac{1}{n^2}\right) \max_{1 \le n \le N_0} a_n \\
&\le \prod_{n=1}^\infty \left(1+\frac{1}{n^2}\right) \max_{1 \le n \le N_0} a_n,
\end{alignat*}
which is finite, and so in contradiction with our assumption.
\end{proof}

\section{Proofs of the main theorem}

We now prove our main theorem. We argue by contradiction. A preliminary observation, which is used at several points in the proof, is a lower bound on the tail of the defining series. Thus, we let.
$$
\gamma = \sum_{n=1}^{\infty} \frac{1}{\alpha_n},
$$
and assume that $\gamma$ is algebraic of degree at nost $D$. Also, let 
$$
\gamma_N = \sum_{n=1}^{N} \frac{1}{\alpha_n}
$$
be the $N$'th partial sum of the series. Clearly, this is an algebraic number, and we estimate its degree and Mahler measure by Lemmas \ref{lem:elementary} and \ref{lem:reciprocal}. First,
$$
\deg(\gamma_N) \le \prod_{i=n}^N \deg\left(\frac{1}{\alpha_n}\right) = \prod_{i=n}^N \deg\left(\alpha_n \right) =: D_N.
$$
For the Mahler measure, we get that
\begin{multline*}
M(\gamma_N) = H(\gamma_N)^{D_N} \le \left(2^N \prod_{n=1}^N H \left(\frac{1}{\alpha_n}\right)\right)^{D_N} \\
= \left(2^N \prod_{n=1}^N H \left(\alpha_n\right)\right)^{D_N} \le \left(2^N \prod_{n=1}^N  \house{\alpha_n}\right)^{D_N},
\end{multline*}
where the latter equality follows from Lemma \ref{lem:height_house_measure}. 

Now, let
$$
\gamma(N) = \sum_{n=N+1}^\infty \frac{1}{\alpha_n} = \gamma - \gamma_N.
$$
Note that by our assumption that $\Re(\alpha_n) > 0$ or $\Im(\alpha_n) > 0$ for all $n\in \mathbb{N}$, $\gamma(N) \neq 0$ for all $N \in \mathbb{N}$. Hence, using the above estimates and Lemma \ref{lem:liouville},
\begin{alignat*}{2}
\vert(\gamma(N))\vert &= \vert \gamma - \gamma_N \vert \\
&\ge \frac{1}{2^{\deg(\gamma)\deg(\gamma_N)}M(\gamma)^{\deg(\gamma_N)}M(\gamma_N)^{\deg(\gamma)}}\\
&\ge \frac{1}{2^{D \cdot D_N}M(\gamma)^{D_N} \left(2^N \prod _{n=1}^N  \house{\alpha_n}\right)^{DD_N}}\\
& = \frac{1}{\left(2^{N +1}H(\gamma) \prod _{n=1}^N  \house{\alpha_n}\right)^{DD_N}}.
\end{alignat*}
The upshot is the following critical estimate, valid for all $N \in \mathbb{N}$,
\begin{equation}
\label{eq:critical}
\vert(\gamma(N))\vert \left(2^{N +1}H(\gamma) \prod _{n=1}^N  \house{\alpha_n}\right)^{D D_N} \ge 1.
\end{equation}

At this point, the path to a contradiction is clear. All we need to do is to ensure that our condition on the growth of the $\alpha_n$ implies that equation \eqref{eq:critical} is violated for arbitrarily large values of $N$. This is accomplished by obtaining an upper bound on $\vert \gamma(N) \vert$. Note that
\begin{equation}
\label{eq:house}
\vert \gamma(N) \vert = \left\vert \sum_{n=N+1}^\infty \frac{1}{\alpha_n} \right\vert \le \sum_{n=N+1}^\infty \frac{1}{\vert \alpha_n \vert} = \sum_{n=N+1}^\infty \frac{1}{\house{\alpha_n}}.
\end{equation}
Consequently, in order to arrive at a contradiction, it suffices to prove that 
\begin{equation}
\label{eq:general}
\sum_{n=N+1}^\infty \frac{1}{\house{\alpha_n}} \left(2^{N +1}H(\gamma) \prod _{n=1}^N  \house{\alpha_n}\right)^{D D_N} < 1
\end{equation}
for arbitrarily large values of $N$. 

The final inequality to be derived depend only on properties of the series $\sum_{n=N+1}^\infty \frac{1}{\house{\alpha_n}}$, a series of reciprocals of increasing real numbers, and so the completion of the proof is just a matter of calculus and in particular application of Lemmas \ref{lem:calculus1} and \ref{lem:calculus2}. 

In the following, we will assume that the degree of the $\alpha_n$ is bounded by some $d \in \mathbb{N}$. Consequently, we may replace $D_N$ in inequality to \eqref{eq:general} by $d^N$ and proceed to derive a contradiction to this statement. We will discuss this restriction further in the final section of the paper. In brief, what is required to proceed is an upper bound on the geometric growth rate of the product of the degrees.

Assume that $\house{\alpha_n} \ge n^{1+\epsilon}$, so that by Lemma \ref{lem:calculus1},
$$
\sum_{n= N+1}^\infty \frac{1}{\house{\alpha_n}}  < \frac{2+1/\epsilon}{\house{\alpha_{N+1}}^{\epsilon/(1+\epsilon)}}.
$$
To proceed, we need to find a lower bound on $a_{N+1}$ for infinitely many $N$. We split the remainder of the proof into three cases.

\emph{Case I:} Suppose that 
$$
\limsup_{n \rightarrow \infty} \house{\alpha_n}^{\frac{1}{D^n \prod_{i=1}^{n-1}((d+1)^i + d+1)}} = \infty.
$$
By Lemma \ref{lem:calculus2}, for infinitely many values of $N$, 
$$
 \house{\alpha_{N+1}}^{\frac{1}{D^{N+1}\prod_{i=1}^{N}((d+1)^i + d+1)}} \ge \left(1 + \frac{1}{N^2}\right) \max_{1 \le n \le N}  \house{\alpha_n}^{\frac{1}{D^n\prod_{i=1}^{n-1}((d+1)^i + d+1)}}.
 $$
 Note that 
 $$
 \log\left(1 + \frac{1}{N^2}\right) \ge \frac{2N^2 - 1}{2N^4},
 $$
 and that
 $$
 \prod_{n=1}^N((d+1)^n + d + 1) \frac{2N^2 - 1}{2N^4} \ge \log 2 (d+1)^N,
 $$
 so that 
 $$
 \left(1 + \frac{1}{N^2}\right)^{D^{N+1}\prod_{i=1}^{N}((d+1)^i + d+1)} > 2^{D^{N+1}(d+1)^N}.
 $$
 
 Hence, we find that for these infinitely many values of $N$,
 \begin{alignat*}{2}
 \house{\alpha_{N+1}} &\ge 2^{D^{N+1}(d+1)^N} \\
 &\quad \left(\max_{1 \le n \le N}  \house{\alpha_n}^{\frac{1}{D^n\prod_{i=1}^{n-1}((d+1)^i + d+1)}}\right)^{D^{N+1}\prod_{i=1}^{N}((d+1)^i + d+1)} \\
 & \ge \left( 2^{D^N} \house{\alpha_N}\right)^{D (d+1)^N} \\
 &\quad \left(\max_{1 \le n \le N}  \house{\alpha_n}^{\frac{1}{D^n\prod_{i=1}^{n-1}((d+1)^i + d+1)}}\right)^{D^{N+1}(d+1)\prod_{i=1}^{N-1}((d+1)^i + d+1)} \\
  & = \left( 2^{D^N} \house{\alpha_N}\right)^{D (d+1)^N} \\
 &\quad \left(\max_{1 \le n \le N}  \house{\alpha_n}^{\frac{1}{D^n\prod_{i=1}^{n-1}((d+1)^i + d+1)}}\right)^{D^{N+1}(d+1)^N\prod_{i=1}^{N-2}((d+1)^i + d+1)} \\
  &\quad \left(\max_{1 \le n \le N}  \house{\alpha_n}^{\frac{1}{D^n\prod_{i=1}^{n-1}((d+1)^i + d+1)}}\right)^{D^{N+1}(d+1)^2\prod_{i=1}^{N-2}((d+1)^i + d+1)} \\
 &\ge \cdots \ge \\
 &\ge \left(2^{D^N}\prod_{n=1}^N \house{\alpha_n}^{D^{N-n}}\right)^{D (d+1)^N}.
 \end{alignat*}
 Consequently, for these values of $N$,
 $$
 \sum_{n= N+1}^\infty \frac{1}{\house{\alpha_n}} \le (2+1/\epsilon) \left(2^{D^N}\prod_{n=1}^N \house{\alpha_n}^{D^{N-n}}\right)^{-(D (d+1)^N)\epsilon/(1+ \epsilon)}.
$$
Inserting this into the right hand side of \eqref{eq:general}, we find that
\begin{alignat*}{2}
\sum_{n=N+1}^\infty &\frac{1}{\house{\alpha_n}} \left(2^{N +1}H(\gamma) \prod _{n=1}^N  \house{\alpha_n}\right)^{D d^N} \\
&\le  (2+1/\epsilon) \left(2^{D^N} \prod_{n=1}^N \house{\alpha_n}^{D^{N-n}}\right)^{-(D (d+1)^N)\epsilon/(1+ \epsilon)}\left(2^{N +1}H(\gamma) \prod _{n=1}^N  \house{\alpha_n}\right)^{D d^N}.
\end{alignat*}
As $N$ can be arbitrarily large, this proves \eqref{eq:general} in the first case.

In the following, we will suppose 
$$
\limsup_{n \rightarrow \infty} \house{\alpha_n}^{\frac{1}{D^n \prod_{i=1}^{n-1}((d+1)^i + d+1)}} < \infty,
$$
but 
$$
\limsup_{n \rightarrow \infty} \house{\alpha_n}^{\frac{1}{D^n \prod_{i=1}^{n-1}(d^i + d)}} = \infty.
$$
By the first of these conditions, 
\begin{equation}
\label{eq:upper}
\house{\alpha_n} < 2^{(d+1)^{n^2}},
\end{equation}
for $n$ sufficiently large. We now deal with the remaining two cases.

\emph{Case II:} Suppose for this case that 
\begin{equation}
\label{eq:lower1}
\house{\alpha_n} \ge 2^{n}
\end{equation}
for all $n$ sufficiently large. We require an upper estimate on the tail, $\gamma(N)$.

For $N$ sufficiently large, write
$$
\vert \gamma(N)\vert \le \sum_{n=N+1}^\infty \frac{1}{\house{\alpha_n}} = \sum_{n=N+1}^{[\log_2 \house{\alpha_{N+1}}+1]} \frac{1}{\house{\alpha_n}} + \sum_{n =[\log_2 \house{\alpha_{N+1}}+1]+1}^\infty \frac{1}{\house{\alpha_n}}.
$$
We estimate the first sum by the maximum value times the number of summands, so that 
$$
\sum_{n=N+1}^{[\log_2 \house{\alpha_{N+1}}+1]} \frac{1}{\house{\alpha_n}} \le \frac{\log_2 \house{\alpha_{N+1}}+1}{\house{\alpha_{N+1}}},
$$
and the second sum by an integral using \eqref{eq:lower1}, so that
$$
\sum_{n =[\log_2 \house{\alpha_{N+1}}+1]+1}^\infty \frac{1}{\house{\alpha_n}} \le \int_{[\log_2 \house{\alpha_{N+1}}+1]}^\infty \frac{\textnormal{d}x}{2^x} \le \frac{1/\log 2}{\house{\alpha_{N+1}}}.
$$
In total, for $N$ large enough, we have on applying \eqref{eq:upper}
\begin{equation}
\label{eq:case2}
\vert \gamma(N) \vert \le \frac{2 \log_2 \house{\alpha_{N+1}}}{\house{\alpha_{N+1}}} < \frac{2 (d+1)^{(N+1)^2}}{\house{\alpha_{N+1}}} \le \frac{2^{dN^2}}{\house{\alpha_{N+1}}}.
\end{equation}

To arrive at a contradiction to \eqref{eq:critical}, we need a lower bound on $\house{\alpha_{N+1}}$ valid for $N$ arbitrarily large. Noting that 
$$
\log\left(1 + \frac{1}{N^2}\right) \prod_{n=1}^N (d^n + d) \ge \log 2 N^2 d^N,
$$
we may argue by Lemma \ref{lem:calculus2} as in the proof of case I to find that for infinitely many $N$,
\begin{alignat*}{2}
 \house{\alpha_{N+1}} &\ge 2^{D^{N+1}N^2 d^N} \left(\max_{1 \le n \le N}  \house{\alpha_n}^{\frac{1}{D^n\prod_{i=1}^{n-1}(d^i + d)}}\right)^{D^{N+1}\prod_{i=1}^{N}(d^i + d)} \\
 & \ge \left( 2^{D^NN^2} \house{\alpha_N}\right)^{D d^N} \left(\max_{1 \le n \le N}  \house{\alpha_n}^{\frac{1}{D^n\prod_{i=1}^{n-1}(d^i + d)}}\right)^{D^{N+1}d\prod_{i=1}^{N-1}(d^i + d)} \\
& = \left( 2^{D^NN^2} \house{\alpha_N}\right)^{D d^N} \left(\max_{1 \le n \le N}  \house{\alpha_n}^{\frac{1}{D^n\prod_{i=1}^{n-1}(d^i + d)}}\right)^{D^{N+1}d^N\prod_{i=1}^{N-2}(d^i + d)} \\
&\quad\quad \left(\max_{1 \le n \le N}  \house{\alpha_n}^{\frac{1}{D^n\prod_{i=1}^{n-1}(d^i + d)}}\right)^{D^{N+1}d^2\prod_{i=1}^{N-2}(d^i + d)} \\
&\ge \cdots \ge \\
&\ge \left(2^{D^NN^2}\prod_{n=1}^N \house{\alpha_n}^{D^{N-n}}\right)^{D d^N}.
 \end{alignat*}

In combination with \eqref{eq:case2}, we now find that for these infinitely many values of $N$,
\begin{alignat*}{2}
\vert \gamma(N) \vert &\left(2^{N +1}H(\gamma) \prod _{n=1}^N  \house{\alpha_n}\right)^{D d^N} \\
&\le 2^{dN^2} \left(2^{D^NN^2}\prod_{n=1}^N \house{\alpha_n}^{D^{N-n}}\right)^{- D d^N} \left(2^{N +1}H(\gamma) \prod _{n=1}^N  \house{\alpha_n}\right)^{D d^N}\\
& \le C 2^{d N^2 + D d^N(N+1) - D^{N+1}d^N N^2},
\end{alignat*}
where $C > 0$ is a constant such that 
$$
\left(\frac{H(\gamma) \prod _{n=1}^N  \house{\alpha_n}}{\prod_{n=1}^N \house{\alpha_n}^{D^{N-n}}}\right)^{D d^N} \le C,
$$
which may clearly be chosen. Unless $d=D=1$, we may choose $N$ so large that this is $< 1$, which contradicts \eqref{eq:critical}. When $d=D=1$, we are in the case considered by Erd\H{o}s \cite{MR539489}, and we already have the theorem.

\emph{Case III:} Assume now that \eqref{eq:lower1} fails for infinitely many values of $n$, i.e. that 
\begin{equation}
\label{eq:lower2}
\house{\alpha_n} < 2^{n},
\end{equation}
for infinitely many values of $n$. By assumption, for any $B$ large enough (to be fixed), there are infinitely many $n$ such that 
\begin{equation}
\label{eq:lower3}
\house{\alpha_n} \ge 2^{B D^n \prod_{i=1}^{n-1}(d^i + d)}.
\end{equation}
Let $s$ be the least natural number satisfying this inequality, and let $k \in \{1, \dots s\}$ be the largest number satisfying \eqref{eq:lower2}. Finally, by Lemma \ref{lem:calculus2}, there are infinitely many $n$ for which
$$
\house{\alpha_n} > \left(\left(1+\frac{1}{(n-1)^2}\right) \max_{1 \le i \le n-1} \house{\alpha_i}^{1/(D^i\prod_{j=1}^{i-1}(d^j + d))}\right)^{D^n \prod_{j=1}^{n-1} (d^j + d)}.
$$
Let $r \ge k$ be the smallest integer satisfying this inequality.

We claim first that $r \le s$. To see this, suppose for a contradiction that $s < r$. By choice of $r$, for any $t \in  \{k+1, \dots, r-1\}$,
$$
\house{\alpha_t} \le \left(\left(1+\frac{1}{(t-1)^2}\right) \max_{k \le i \le t-1} \house{\alpha_i}^{1/(D^i\prod_{j=1}^{i-1}(d^j + d))}\right)^{D^t \prod_{j=1}^{t-1} (d^j + d)}.
$$
Applying first \eqref{eq:lower3} and the definition of $s$, and subsequently the above inequality several times, noting that $k < s < r$, we find that
\begin{alignat}{2}
2^B &\le \house{\alpha_s}^{1/(D^s \prod_{j=1}^{s-1} (d^j + d))} \notag \\
&\le \left(1+\frac{1}{(s-1)^2}\right) \max_{k \le n \le s-1} \house{\alpha_n}^{1/(D^i\prod_{j=1}^{n-1}(d^j + d))} \notag \\
&\le \cdots \le \left(\prod_{j=k+1}^{s-1} \left(1 + \frac{1}{(j-1)^2}\right)\right) a_k^{1/(D^k\prod_{j=1}^{k-1}(d^j + d))} \notag \\
&< 2\prod_{j=k+1}^{s-1} \left(1 + \frac{1}{(j-1)^2}\right) = \frac{2 \sinh(\pi)}{\pi} < 8. \label{eq:intermediate}
\end{alignat}
On choosing $B \ge 3$, we obtain the desired contradiction, and we have $r \le s$.

To proceed, note that 
\begin{multline*}
\prod_{n=1}^{r-1}(d^n + d) \log\left(1 + \frac{1}{(r-1)^2}\right) \ge d^{\frac{1}{2} (r-1)^2 + \frac{1}{2} (r-1)} \frac{2(r-1)^2 - 1}{2(r-1)^4} \\
\ge d^{4r} \log 2.
\end{multline*}
Arguing by Lemma \ref{lem:calculus2} as in the preceding cases,
\begin{alignat}{2}
\notag \house{\alpha_{r}} &\ge 2^{D^{r}d^{4r}} \left(\max_{1 \le n \le r-1}  \house{\alpha_n}^{\frac{1}{D^n\prod_{i=1}^{n-1}(d^i + d)}}\right)^{D^{r}\prod_{i=1}^{r-1}(d^i + d)} \\
\label{eq:alphar} &\ge 2^{D^rd^{4r}} \left(\prod_{n=1}^{r-1} \house{\alpha_n}\right)^{D^r d^{r-1}}.
\end{alignat}
With the argument leading to \eqref{eq:intermediate}, we find that for $t \in {k+1, \dots, r-1}$,
$$
\house{\alpha_t}^{1/(D^t \prod_{j=1}^{t-1} (d^j + d))} \le 8.
$$
Since the sequence $\house{\alpha_n}$ increases, the above estimates imply by choice of $k$ that
$$
\prod_{i=1}^{r-1} \house{\alpha_i} = \prod_{i=1}^{k} \house{\alpha_i}\prod_{i=k+1}^{r-1} \house{\alpha_i} \le 2^{k^2} \prod_{i=k+1}^{r-1} 8^{D^i \prod_{j=1}^{i-1} (d^j + d)}.
$$
Note that this immediately implies that  for some $C > 0$ independent of $k$, 
\begin{equation}
\label{eq:product}
\prod_{i=1}^{r-1} \house{\alpha_i} \le C^{D^{r-1} \prod_{j=1}^{r-2} (d^j + d)}.
\end{equation}

To finish the proof, we consider $\vert \gamma(r-1) \vert$ by splitting it into three sums,
$$
\vert \gamma(r-1) \vert \le \sum_{n=r}^\infty \frac{1}{\house{\alpha_n}} = \sum_{n=r}^{[\log_2\house{\alpha_r}]}\frac{1}{\house{\alpha_n}} + \sum_{n=[\log_2\house{\alpha_r}+1]}^{s-1} \frac{1}{\house{\alpha_n}} + \sum_{n=s}^\infty \frac{1}{\house{\alpha_n}}.
$$
For the first sum, as $\house{\alpha_n}$ increases, 
$$
\sum_{n=r}^{[\log_2\house{\alpha_r}]}\frac{1}{\house{\alpha_n}} \le \frac{\log_2\house{\alpha_r}}{\house{\alpha_r}}.
$$
In the summation range of the second sum, $\house{\alpha_n} \ge 2^n$, so
$$
 \sum_{n=[\log_2\house{\alpha_r}+1]}^{s-1} \frac{1}{\house{\alpha_n}} \le  \sum_{n=[\log_2\house{\alpha_r}+1]}^\infty \frac{1}{2^n} = \frac{2}{2^{[\log_2\house{\alpha_r}+1]}} \le \frac{2}{\house{\alpha_r}} \le\frac{\log_2 \house{\alpha_r}}{\house{\alpha_r}} .
 $$
 The third sum is estimated by Lemma \ref{lem:calculus1},
 $$
\sum_{n=s}^\infty \frac{1}{\house{\alpha_n}} \le \frac{2+1/\epsilon}{\house{\alpha_s}^{\epsilon/(1+\epsilon)}}.
$$
Hence,
$$
\vert \gamma(r-1) \vert \le \frac{2\log_2 \house{\alpha_r}}{\house{\alpha_r}} + \frac{2+1/\epsilon}{\house{\alpha_s}^{\epsilon/(1+\epsilon)}}.
$$

Applying \eqref{eq:upper} and \eqref{eq:alphar} to the first summand,
\begin{multline*}
 \frac{2\log_2 \house{\alpha_r}}{\house{\alpha_r}} \le  \frac{2(d+1)^{r^2}}{\house{\alpha_r}} 
 < (d+1)^{r^2} 2^{1-D^rd^{4r}} \left(\prod_{n=1}^{r-1} \house{\alpha_n}\right)^{-D^r d^{r-1}} \\<  2^{-D^rd^{3r}} \left(\prod_{n=1}^{r-1} \house{\alpha_n}\right)^{-D^r d^{r-1}}.
 \end{multline*}
 For the second summmand, by choice of $s$,
$$
\frac{2+1/\epsilon}{\house{\alpha_s}^{\epsilon/(1+\epsilon)}} \le \frac{2+1/\epsilon}{(2^{BD^s\prod_{i=1}^{s-1}(d^i + d)})^{\epsilon/(1+\epsilon)}}.
$$
We claim that these inequalities will contradict \eqref{eq:critical} whenever $r$ is large enough.

To see this, note that
\begin{multline*}
\vert(\gamma(r-1))\vert \left(2^{r}H(\gamma) \prod _{n=1}^{r-1}  \house{\alpha_n}\right)^{D d^{r-1}} 
 < \Bigg(2^{-D^rd^{3r}} \left(\prod_{n=1}^{r-1} \house{\alpha_n}\right)^{-D^r d^{r-1}} \\
 + \frac{2+1/\epsilon}{(2^{BD^s\prod_{i=1}^{s-1}(d^i + d)})^{\epsilon/(1+\epsilon)}}\Bigg) \left(2^{r}H(\gamma) \prod _{n=1}^{r-1}  \house{\alpha_n}\right)^{D d^{r-1}}.
\end{multline*}

Now, 
$$
2^{-D^rd^{3r}} \left(\prod_{n=1}^{r-1} \house{\alpha_n}\right)^{-D^r d^{r-1}} \left(2^{r}H(\gamma) \prod _{n=1}^{r-1}  \house{\alpha_n}\right)^{D d^{r-1}}
$$
clearly tends to $0$ as $r$ increases. For the second summand, we use \eqref{eq:product}, so that 
\begin{multline*}
\frac{2+1/\epsilon}{(2^{BD^s\prod_{i=1}^{s-1}(d^i + d)})^{\epsilon/(1+\epsilon)}} \left(2^{r}H(\gamma) \prod _{n=1}^{r-1}  \house{\alpha_n}\right)^{D d^{r-1}} \\
\le \frac{2+1/\epsilon}{(2^{BD^s\prod_{i=1}^{s-1}(d^i + d)})^{\epsilon/(1+\epsilon)}} \left(2^{r}H(\gamma) C^{D^{r-1} \prod_{j=1}^{r-2} (d^j + d)} \right)^{D d^{r-1}},
\end{multline*}
which evidently also tends to zero as $r$ increases, on noting that $s \ge r$.

We now have a contradiction to \eqref{eq:critical} for a single value of $r$, provided the value of $r$ produced above is large enough. However, we need infinitely many counterexamples to prove the result. This is easy. Recall that $k \le r \le s$, and note that if we increase $B$, $s$ will increase. Hence, $k$ will increase, considering how $k$ was chosen. It follows that on increasing $B$, we will obtain infinitely many values of $r$ so that the above holds. This is in contradiction with \eqref{eq:critical} and completes the proof \qed

\section{Concluding remarks}

We end our paper with some concluding remarks on possible extensions of our result. As the reader will have noticed, we have made quite a few assumptions in our proof, which could be weakened. These have been made in order to simplify the statement and presentation. We comment on these individually.

The first observation is concerned with the $\alpha_n$ being algebraic integers. In fact, this is needed only in the application of Lemma \ref{lem:height_house_measure}. The assumption can be replaced by any inequality relating $\vert \alpha_n \vert$ to $M(\alpha_n)$. Note that in the same spirit, we could replace the assumption that $\vert \alpha_n \vert = \house{\alpha_n}$ with a similar relation. Of course, this would change the growth condition for the \emph{limsup}.

The second observation concerns the assumption on the upper bound on $\deg(\alpha_n)$ being fixed. In our statement of the main theorem, we introduce a uniform bound on the degrees of the numbers $\alpha_n$. This allows us to replace $D_N = \prod_{n=1}^N \deg(\alpha_n)$ in \eqref{eq:critical} by $d^N$ throughout, which simplifies calculations enormously. In the interest of clarity, we will not pursue this any further here, but with additional care, it should be possible to obtain a growth estimate depending on $D_N$ rather than $d^N$.

The third observation concerns the assumption that $\Re(\alpha_n)> 0$ or $\Im(\alpha_n) > 0$ for all $n \in \mathbb{N}$. The only purpose of this assumption is to ensure that $\gamma(N) \neq 0$ for all $N \in \mathbb{N}$. Any other assumption ensuring this would be sufficient to ensure the conclusion of Theorem \ref{thm:main} with the remaining assumptions being unaffected. Nevertheless, it is entirely possible to construct a sequence of algebraic integers satisfying all the other assumptions of the theorem, such that the resulting series is rational.

Our final observation concerns the special case when all $\alpha_n$ are Pisot or Salem numbers, so that $M(\alpha_n) = \house{\alpha_n}$. In our proof of the main theorem, we applied only the first inequality in Lemma \ref{lem:height_house_measure} in our derivation of \eqref{eq:critical}. While this inequality is true and best possible for all algebraic integers, if one specialises to all the $\alpha_n$ being Pisot or Salem numbers, a stronger inequality is obtained, namely that
\begin{equation}
\label{eq:critical_PS}
\vert(\gamma(N))\vert \left(2^{N +1}H(\gamma) \prod _{n=1}^N  \house{\alpha_n}^{1/\deg(\alpha_n)}\right)^{D D_N} \ge 1.
\end{equation}
This would lead to a weaker assumption on the growth of the sequence of $\house{\alpha_n}$.

\providecommand{\bysame}{\leavevmode\hbox to3em{\hrulefill}\thinspace}
\providecommand{\MR}{\relax\ifhmode\unskip\space\fi MR }
\providecommand{\MRhref}[2]{%
  \href{http://www.ams.org/mathscinet-getitem?mr=#1}{#2}
}
\providecommand{\href}[2]{#2}

\end{document}